\documentclass[12pt,reqno]{amsart}
\usepackage[utf8]{inputenc}
\usepackage[T1]{fontenc}
\usepackage[usenames, dvipsnames]{color}
\usepackage{ulem}

\usepackage{dsfont, amsfonts, amsmath, amssymb,amscd, stmaryrd, latexsym, amsthm, dsfont}
\usepackage[frenchb,english]{babel}
\usepackage{enumerate}
\usepackage{longtable}
\usepackage{geometry}
\usepackage{float}
\usepackage{tikz}
\usetikzlibrary{shapes,arrows}
\usepackage{float}
\usepackage{tikz}
\usepackage{xypic}
\usetikzlibrary{shapes,arrows}

\newtheorem{theorem}{Theorem}[section]
\newtheorem{lemma}[theorem]{Lemma}
\newtheorem{proposition}[theorem]{Proposition}
\newtheorem{corollary}[theorem]{Corollary}

\theoremstyle{remark}
\newtheorem{remark}[theorem]{\bf Remark}
\newtheorem{question}[theorem]{\bf Question}
\newtheorem{example}[theorem]{\bf Example}
\newtheorem{definition}[theorem]{\bf Definition}

\usepackage[pagebackref]{hyperref}
\renewcommand*{\backref}[1]{}\renewcommand*{\backrefalt}[4]{\ifcase #1 (\tt not cited)\or (\tt cited on page~#2)\else (\tt cited on pages~#2)\fi}

\usepackage{hyperref}
\hypersetup{
	colorlinks=true,
	urlcolor=blue,
	citecolor=blue}
\def\intR{\mathrm{Int^R}}
\def\intRB{\mathrm{Int}^\mathrm{R}_B}
\def\dim{\mathrm{dim}}
\begin{document}
\selectlanguage{english}
\title[]{On rings of  integer-valued rational functions}
	
\author[M. M. Chems-Eddin]{M. M. Chems-Eddin}
\address{Mohamed Mahmoud CHEMS-EDDIN: Department of Mathematics, Faculty of Sciences Dhar El Mahraz, Sidi Mohamed Ben Abdellah University, Fez,  Morocco}
\email{2m.chemseddin@gmail.com}
 
\author[B. Feryouch]{B. Feryouch}
\address{Badr Feryouch: Department of Mathematics, Faculty of Sciences Dhar El Mahraz, Sidi Mohamed Ben Abdellah University, Fez,  Morocco}
\email{badr.feryouch@usmba.ac.ma}
 
\author[H. Mouanis]{H.  Mouanis}
\address{Hakima  Mouanis: Department of Mathematics, Faculty of Sciences Dhar El Mahraz, Sidi Mohamed Ben Abdellah University, Fez,  Morocco}
\email{hmouanis@yahoo.fr}

\author[A. Tamoussit]{A. Tamoussit}
\address{Ali TAMOUSSIT: Department of Mathematics, The Regional Center for Education and Training Professions Souss-Massa,  Inezgane, Morocco}
\email{a.tamoussit@crmefsm.ac.ma ; tamoussit2009@gmail.com}

\subjclass[2020]{13F05, 13F20, 13B30, 16D40.}
\keywords{Integer-valued  rational functions, prime ideals, (faithfully) flat modules.}
	
\maketitle
\begin{abstract} 
Let $D\subseteq B$ be an extension of integral domains and  $E$ a subset of the quotient field of $D$.  We introduce the ring of \textit{$D$-valued $B$-rational functions on $E$}, denoted by $\intRB(E,D)$, which naturally extends the concept of integer-valued polynomials, defined as $\intRB(E,D):=\lbrace f \in B(X);\; f(E)\subseteq D\rbrace.$ {The notion of $\intRB(E,D)$ boils down to the usual notion of integer-valued rational functions when the subset $E$ is infinite}. In this paper, we aim to investigate various properties of these rings, such as prime ideals, localization, and the module structure. Furthermore, we study the transfer of some ring-theoretic properties from $\intR(E,D)$ to $D$. 
\end{abstract}

\section*{Introduction}
Let $D$ be an integral domain and $K$ its quotient field. The set denoted and defined by:   
\begin{eqnarray}\label{intD}
\mathrm{Int}(D):= \{f\in K[X];\;  f(D)\subseteq D\},  
\end{eqnarray}
is known to be a commutative $D$-algebra  and is called \textit{the ring of integer-valued polynomials over} $D$. The study of these rings dates back to 1919, when P\'{o}lya~\cite{P19} and Ostrowski~\cite{O19} introduced and  investigated $\mathrm{Int}(\mathcal{O}_K)$,  where $\mathcal{O}_K$ denotes the ring of integers of a number field $K$. This investigation gave rise to an interesting subfield of Algebraic Number Theory called P\'{o}lya Theory, which remains a very active area of research.  Later, in 1970, Cahen and Chabert generalized  these rings, by considering an arbitrary integral domain $D$ instead of $\mathcal{O}_K$ (cf. \eqref{intD}). It is noteworthy  that the theory of these rings plays an important role in the Multiplicative Ideal Theory and has led to some very surprising results. Particularly, it is a very rich source of examples and counterexamples.

\medskip

In 1991, Anderson \textit{et al.} \cite{AAZ91} treated the ring of \textit{$D$-valued $B$-polynomials}, defined as follows:  $$\mathrm{Int}_B(D):=\{f\in B[X];\; f(D)\subseteq D\},$$  where $D\subseteq B$ is an extension of integral domains, as a generalization of rings of integer-valued polynomials.  Thereafter, in 1993, Cahen~\cite{C93} considered another generalization of $\mathrm{Int}(D)$, namely, the ring of \textit{integer-valued polynomials over a subset} $E$ of $K$,  defined by: $$\mathrm{Int}(E,D):=\{f\in K[X];\; f(E)\subseteq D\},$$ where $D$ is an integral domain with quotient field $K$.  A natural generalization of these rings is defined by: $$\intR(E,D):=\{f\in K(X);\; f(E)\subseteq D\},$$	 which is called \textit{the ring of integer-valued rational functions on} $E$ \textit{over} $D.$ The ring $\mathrm{Int^R}(E,D)$ is investigated in  \cite[Chapter X]{C97}, and recently, in 2023, it was studied by B. Liu in  his Ph.D. thesis \cite{BL23}.

\medskip

In 2020, the fourth named author introduced the concept of \textit{$D$-valued $B$-polynomials on $E$} in \cite{AT2021}. This ring is formally defined as:
$$\mathrm{Int}_B(E,D):=\{f\in B[X];\; f(E)\subseteq D\},$$
where $D \subseteq B$ represents an extension of integral domains, and $E$ is a subset of the quotient field of $D$. These rings are considered to be  a generalization of both $\mathrm{Int}_B(D)$ and $\mathrm{Int}(E,D)$. A recent study investigating various properties of these rings is presented in \cite{CT23}.

\medskip

Let us now focus on giving a quick review of the studies established concerning the ring of integer-valued rational functions. 
The study of these rings started in 1978 with  P.-J. Cahen in his work \cite{C78}, where he introduced and studied them. Specifically, he showed that for any discrete valuation domain $V$ with quotient field $K$ and any subset $E$ of $K$, the ring $\intR(E,V)$ is a localization of $\intR(K,V)$ and it is  Pr\"ufer.  Additionally, he  also studied its prime spectrum. Thereafter, in 1991,  A. Prestel and   C. Rippol \cite{PR91} provided a remarkable result. Explicitly,  they proved that:  for a field $K$ and a non-trivial valuation ring $V$ of $K$ with $\mathfrak{m}$ as its maximal ideal, the equality $\intR(V) = \left(1 + \mathfrak{m}\mathrm{Int}(V)\right)^{-1}\mathrm{Int}(V)$ holds if and only if the completion of $(K,V)$ is either locally compact or algebraically closed. In 1994, K.A. Loper \cite{Lop94} demonstrated that $\intR(D)$ is a Pr\"ufer domain for any monic Pr\"ufer domain $D$. Three years after,  Cahen and  Chabert published the book  \cite{C97} entitled {\it Integer-Valued Polynomials}, devoting  its tenth chapter to the study of the ring $\intR(E,D)$.  Shortly after, in 1998, Cahen and Loper \cite{CL98} were interested in the  classification of integral domains $D$ such that $\intR(E,D)$ is either Pr\"ufer or Bézout, with special attention given to the case  $D=V$ is a valuation domain. In the same year, Gabelli and  Tartarone in \cite{SF20}  investigated some properties of $\intR(D)$, including localization and the Bézout property. Recently,  in 2022, M.H. Park  \cite{MP22} investigated the problem of when $\intR(E,D)$ is a globalized pseudo-valuation domain. Among other things, she showed that if $D$ is a pseudo-valuation domain with associated valuation domain $V$ and $E$ is a precompact subset of the quotient field $K$, then $\intR(E,D)$  is a globalized pseudo-valuation domain with associated Pr\"ufer domain $\intR(E,V)$, and moreover $\intR(E,D)$ is a localization of $\mathrm{Int}(E,D)$. Subsequently,  in 2023,  B. Liu in his Ph.D. thesis \cite{BL23} investigated many questions concerning these rings, and more precisely whether $\intR(V)$ is a   Pr\"ufer domain or not for valuation domains $V$.  Despite the many mouth of studies concerning these rings, the community still in need to invest  more and more efforts for studying and discovering the rings $\intR(E,D)$. Surprisingly, till now, no investigation has been conducted on the module structure of $\intR(E,D)$ as either a $D$-module or a $D[X]$-module.
    
\medskip

We note that this ring has been studied in the literature much less than the classical ring of integer-valued polynomials, due to several intrinsic difficulties related to it. Therefore, as a continuation of the previous investigations,  we introduce and then study the ring of \textit{$D$-valued $B$-rational functions on $E$}, which  defined as follows: 	 $$\intRB(E,D):=\{f\in B(X);\;f(E)\subseteq D\},$$
where $D\subseteq B$ is an extension of integral domains, $E$ is a subset of the quotient field of $D$, and {$B(X)$ denotes the set of all rational functions over $B.$} It is clear that $\mathrm{Int}_B(E,D)\subseteq\intRB(E,D)\subseteq B(X),$ and that the ring $\intRB(E,D)$ is an extension of all previously defined rings.  When $B$ is the quotient field of $D$, we simply write $\intR(E,D)$ (resp.,  $\intR(D)$) for $\intRB(E,D)$ (resp., $\intRB(D,D)$ and  $\intRB(D)$).  Notably, as proved in the first section, $\intRB(E,D)$ and $\intR(E,D)$ coincide when $E$ is infinite. 
	 
\medskip	 	
	 
We now give a brief description of the content of this paper. In Section \ref{sec2}, we start by a first investigation of the  properties of the ring $\intRB(E,D)$. Section \ref{sec3} is devoted to the localization and some ideals of  $\intR(E,D)$. In Section \ref{sec4}, we study the transfer of certain ring-theoretic properties from $\intR(E,D)$ to $D$ (such as  seminormal, ($t$-)finite character, (strong) Mori, P$v$MD, essential, ...). Finally, in Section \ref{sec5},  we give some results on the module structure of  $\intRB(E,D)$. Particularly, we are interested in the flatness and the ($w$-)faithful flatness properties.

\medskip

Throughout this paper, $D$ is an integral domain of quotient field $K$, $E$ is a nonempty subset of $K$ and $B$ is an integral domain containing $D$. Additionally, the symbols $\subset$ and $\subseteq$ denote proper containment and large containment, respectively.

\section{Preliminary definitions and properties}\label{sec2}

We start this section by recalling some definitions and notation. Let $D$ be an integral domain with quotient field $K$. A subset $E$ of $K$ is said to be a \textit{fractional subset} of $D$ if  there exists a nonzero element $d$ of $D$ such that $dE\subseteq D.$  For a nonzero fractional ideal $I$ of $D$, we let $I^{-1}:=\{x\in K;\; xI\subseteq D\}$. On $D$ the $v$-\textit{operation} is defined by $I_v:=(I^{-1})^{-1}$, the $t$-\textit{operation} is defined by $I_t:=\bigcup J_v$, where $J$ ranges over the set of all nonzero finitely generated ideals contained in $I$; and the $w$-\textit{operation} is defined by $I_w:=\{x\in K;\;xJ\subseteq I$ for some nonzero finitely generated ideal $J$ of $D$ with $J^{-1}=D\}$. A nonzero ideal $I$ of $D$ is \textit{divisorial} (or $v$-\textit{ideal}) (resp., $t$-\textit{ideal}, $w$-\textit{ideal}) if $I_v=I$ (resp., $I_t=I$, $I_w=I$).  In general, for each nonzero fractional ideal $I$ of $D$, we have the inclusions $I\subseteq I_w\subseteq I_t\subseteq I_v,$ and then $v$-ideals are $t$-ideals and $t$-ideals themselves are $w$-ideals. If $\star$ denotes either $t$ or $w$, A $\star$-ideal that is also prime is called a $\star$-\textit{prime ideal}, and a $\star$-\textit{maximal ideal} is a maximal ideal among all $\star$-ideals of $D$ and the set of all $\star$-maximal ideals of $D$ is denoted by $\star$-$\mathrm{Max}(D)$. Moreover, we have $t$-$\mathrm{Max}(D)=w$-$\mathrm{Max}(D)$.  Notice that each height-one prime is $t$-prime and when each $t$-prime ideal of $D$ has height-one we say that $D$ has $t$-\textit{dimension one} which we denote by $t$-$\dim(D)=1$; in this case, we have $t$-$\mathrm{Max}(D)=X^1(D)$, where $X^1(D)$ is the set of all height-one prime ideals of $D$.

\medskip

An integral domain $D$ is said to be a \textit{Pr\"{u}fer $v$-multiplication domain} (for short, P$v$MD)  if $D_\mathfrak{m}$ is a valuation domain for each $t$-maximal ideal $\mathfrak{m}$ of $D$. An integral domain $D$ is said to be a $t$-\textit{almost Dedekind domain} if $D_\mathfrak{m}$ is a DVR (here by a DVR we mean a rank-one discrete valuation domain) for each $t$-maximal ideal $\mathfrak{m}$ of $D$. Trivially,  Krull domains and almost Dedekind domains are $t$-almost Dedekind domains and $t$-almost Dedekind domains are P$v$MDs.  An integral domain $D$ is a \textit{strong Mori domain}  (resp., \textit{Mori domain}) if it satisfies the ascending chain condition  on $w$-ideals (resp., $v$-ideals) of $D$. Clearly, Noetherian domains and Krull domains are strong Mori and strong Mori domains are Mori. An integral domain $D$ is said to be \textit{locally Mori} (resp., \textit{almost Krull}) if any localization of $D$ at a maximal ideal is a Mori (resp., a Krull) domain. Obviously, almost Dedekind domains and Krull domains are almost Krull domains, and almost Krull domains themselves are locally Mori. Lastly, an integral domain $D$ is said to be of \textit{finite character} (resp., $t$-\textit{finite character}), if every nonzero element of $D$ belongs to only finitely many maximal ideals (resp., $t$-maximal ideals) of $D$. It is worth noting that Mori domains have $t$-finite character.

\medskip

From the definition of the  ring of $D$-valued $B$-rational functions on $E$ as stated in the introduction, we deduce directly the following: 
\begin{proposition}\label{Comp}
Let $D_1\subseteq D_2$ be two integral domains with the same quotient field $K$ and  $E\subseteq F$  two nonempty subsets of $K$. Let $D_i\subseteq B_i$ be two extensions of integral domains for  $i\in \{1,2\}$. If $B_1 \subseteq B_2$, then $\mathrm{Int}^\mathrm{R}_{B_1} \left(F,D_1\right) \subseteq  \mathrm{Int}^\mathrm{R}_{B_2}\left(E,D_2\right)$.
\end{proposition}
	 
\begin{corollary}The following statements are equivalent:
\begin{enumerate}[$ (1)$]
\item $E\subseteq D$,
\item $\intRB(D) \subseteq \intRB(E,D)$,
\item $D[X] \subseteq\intRB(E,D)$.
\end{enumerate}
\end{corollary}
\begin{proof} 
(1) $\Rightarrow$ (2) Just take $D_1=D_2=D$, $B_1=B_2=B$ and $F=D$ in the previous proposition. 

(2) $\Rightarrow$ (3) It  follows from the trivial inclusion $D[X]\subseteq\intRB(D)$.

(3) $\Rightarrow$ (1) Assume that  $D[X] \subseteq\intRB(E,D)$. As the polynomial $\psi(X)=X$ belongs to $D[X]$, we deduce that $\psi \in \intRB(E,D)$, and hence $\psi(e)\in D$ for each $e\in E$. Therefore, $E\subseteq D$.
\end{proof}
	 
Notice that the previous corollary says  that $\intRB(E,D)$ is an overring of $D[X]$ if and only if $E$ is contained in $D$.

\medskip
	 
We next rewrite \cite[Proposition X.1.4]{C97} as follows: 
\begin{proposition}\label{Prop1.1}
If $E$ is infinite, then, $\intR(E,D)=\mathrm{Int}^\mathrm{R}_L(E,D)$, for any field extension $L$ of $K$.
\end{proposition}
We can extend this proposition to any extension $B$ of $D$. Unlike the well-known case of integer-valued polynomials, where $\mathrm{Int}_B(E,D)\neq \mathrm{Int}(E,D)$ in general (for instance, see Example \ref{ExamD-ring}), the following result shows that $\intRB(E,D)$ and $\intR(E,D)$ coincide when $E$ is infinite. To do this, we need the following lemma.

\begin{lemma}\label{nwe}
If $L$ denotes the quotient field of $B,$ then $B(X)=L(X)$, and hence $\intRB(E,D)= \mathrm{Int}^\mathrm{R}_{L}(E,D).$ 
\end{lemma}
\begin{proof}
The inclusion $B(X)\subseteq L(X)$ is clear. For the reverse inclusion, let $\varphi=\frac{f}{g} \in L(X)$ with $f$ and $g$ are in $L[X]$. Then, there exist two nonzero elements $u$ and $v$ of $D$ such that $uf\in B[X]$ and $vg\in B[X],$ and hence $\varphi=\frac{uvf}{uvg}\in B(X).$ Thus, $L(X)\subseteq B(X)$, and this completes the proof. 
\end{proof}

\begin{proposition}\label{d1}
If $E$ is infinite, then $\intRB(E,D)=\intR(E,D)$. In particular, $\intRB(D)=\intR(D)$.
\end{proposition}	
\begin{proof}
Assume that $E$ is infinite. If $L$ denotes the quotient field of $B,$ it follows from Lemma \ref{nwe} that $\intRB(E,D)= \mathrm{Int}^\mathrm{R}_{L}(E,D),$ and this last ring is simply $\intR(E,D)$ as asserted in Proposition \ref{Prop1.1} because $E$ is infinite. For the particular case, we just take $E=D$ in the proved equality.
\end{proof}

As previously stated in Proposition \ref{d1}, we have $\intRB(D)=\intR(D)$. However, the behavior of $\mathrm{Int}_B(D)$ differs, as illustrated by the next example. To provide such an example, we recall that $D$ is called a $d$-\textit{ring} if $\mathrm{Int^R}(D)=\mathrm{Int}(D)$.

\begin{example}\label{ExamD-ring}
Let us consider the ring $\mathbb{Z}[\frac13]$ as an overring of $\mathbb{Z}.$ It is known that $\mathbb{Z}$ is a $d$-ring by \cite[Example VII.2.2]{C97}, so $\mathrm{Int}^\mathrm{R}_{\mathbb{Z}[\frac13]}(\mathbb{Z})= \intR(\mathbb{Z})=\mathrm{Int}(\mathbb{Z})$ and we also have $\mathrm{Int}(\mathbb{Z})\neq \mathrm{Int}_{\mathbb{Z}[\frac13]}(\mathbb{Z})$ because the polynomial $\psi(X)=\frac{X(X+1)}{2}$  lies in $  \mathrm{Int}(\mathbb{Z})$ but not in $\mathrm{Int}_{\mathbb{Z}[\frac13]}(\mathbb{Z})$.
\end{example}  
  
  The following result gives a necessary condition on $E$ for $\mathrm{Int}_B(E,D)$ (resp., $\intR(E,D)$) to contain a non-constant polynomial.
 \begin{proposition}\label{ProNCP}
If $D^{\prime}$ denotes the integral closure  of $D$ and $B$ is an overring of $D$, then each of the following statements implies the next: 
\begin{enumerate}[$(1)$]
\item  $\mathrm{Int}_B(E,D)$ contains a non-constant polynomial;
\item $\intR(E,D)$ contains a non-constant polynomial;
\item  $E$ is a fractional subset of $D^{\prime}$.
\end{enumerate}
\end{proposition}
	 
\begin{proof}
(1) $\Rightarrow$ (2) follows from the inclusion $\mathrm{Int}_B(E,D) \subseteq \intR(E,D)$.

(2) $\Rightarrow$ (3) 	Let $f=\sum_{0\leq i\leq n}a_{i}X^{i}$ be a polynomial of degree $n\geq 1$ in $\intRB(E,D)$. By multiplying $f$ by a common denominator of its coefficients, we can always suppose that $f\in D[X]$. For each $e\in E$, we have $a_{n}^{n-1}f(e)\in D$. From the equality 
$a_{n}^{n-1}f(e)=(a_{n}e)^{n}+a_{n-1}(a_{n}e)^{n-1}+\dots +a_{0}(a_{n})^{n-1},$ we deduce that $(a_{n}e)^{n}+a_{n-1}(a_{n}e)^{n-1}+\dots +a_{0}(a_{n})^{n-1}-a_{n}^{n-1}f(e)=0,$ which means that $a_ne\in D^{\prime}.$ Hence, $a_{n}E\subseteq D^{\prime}$, and thus $E$ is a fractional subset of $D^{\prime}$.
\end{proof}
\begin{corollary}\label{ry}
If $D$ is integrally closed, then  $\intR(E,D)$ contains a non-constant polynomial if and only if  $E$ is a fractional subset of $D$.
\end{corollary}

The next example illustrates that it is possible to have $\mathrm{Int}_B(E,D)=D$ while $\intR(E,D)\neq D$ for a non-fractional subset $E$ of an integrally closed domain $D$.

\begin{example}
Consider the set  $E=\lbrace 1,\frac{1}{2},\frac{1}{4},\dots \rbrace$. Since $\mathbb{Z}$ is integrally closed and $E$ is not a fractional subset of $\mathbb{Z},$ it follows from Proposition \ref{ProNCP} that $\mathrm{Int}_B(E,\mathbb{Z})=\mathbb{Z},$ while $\intR(E,\mathbb{Z})\neq \mathbb{Z}$ because $\varphi(X)=\frac{1}{X}$ lies in $\intR(E,\mathbb{Z})$ but not in $\mathbb{Z}$.
\end{example}

Now, we provide a necessary and sufficient condition on the subset $E$ of $K$ for $\intR(E,D)$ to contain a  polynomial of degree one. Explicitly, we show that $\intR(E,D)$ contains a polynomial of degree one if and only if $E$ is a fractional subset of $D$.
\begin{proposition}
For any overring $B$ of $D$, the following assertions are equivalent:
\begin{enumerate}[$(1)$]
\item  $\mathrm{Int}_B(E,D)$ contains a polynomial of degree one;
\item $\intR(E,D)$ contains a polynomial of degree one;
\item  $E$ is a fractional subset of $D$.
\end{enumerate}
\end{proposition}

\begin{proof}
(1) $\Rightarrow$ (2) This follows directly from the inclusion $\mathrm{Int}_B(E,D) \subseteq \intR(E,D)$.

(2) $\Rightarrow$ (3) Suppose $f(X)=aX+b$ is a polynomial of degree one in $\intR(E,D)$ with coefficients $a$ and $b$ in $K$. Then there exists a nonzero element $c$ in $D$ such that $cf(X)=caX+cb \in D[X]$. For each $e\in E$, $cf(e)\in D$ implies $cae\in D$. Therefore, $E$ is a fractional subset of $D$.

(3) $\Rightarrow$ (1) follows from the existence of a nonzero element $d\in D$ such that $dE\subseteq D$. Consequently, $f(X)=dX\in \mathrm{Int}_B(E,D)$, and this completes the proof.
\end{proof}

\section{Localization and ideals of  {$\intR(E,D)$}}\label{sec3}
	 
This section focuses on studying some (prime) ideals of $\intR(E,D)$ and investigating the behavior $\intR(E,D)$ under localization.

\medskip

We start by stating the following useful lemma.  

\begin{lemma}\label{Lolem}
For any prime ideal $\mathfrak{P}$ of $\intR(E,D)$, we have  $$D_{\mathfrak{P}\cap D}=\intR(E,D)_{\mathfrak{P}} \cap K.$$
\end{lemma} 
\begin{proof}
This is quite trivial.
\end{proof}
	 
\begin{definition}\label{Pointed}
Using the above notations and assumptions, let $I$ be an ideal of $D$ and $a$ an element of $E$. We define the set $\mathfrak{I}_{I,a}$ as follows:
  $$\mathfrak{I}_{I,a}:=\lbrace \varphi \in \intR(E,D);\; \varphi(a) \in I \rbrace .$$
\begin{enumerate}[$(1)$]   
\item $\mathfrak{I}_{I,a}$ has the structure of an ideal of $\intR(E,D)$.
\item Ideals of $\intR(E,D)$ of this form are called pointed ideals.
\item If $I$ is a prime ideal $\mathfrak{p}$ (resp., maximal ideal $\mathfrak{m}$) of $D$, we use the notation $\mathfrak{P}_{\mathfrak{p},a}$ (resp., $\mathfrak{M}_{\mathfrak{m},a}$) instead of $\mathfrak{I}_{\mathfrak{p},a}$ (resp., $\mathfrak{I}_{\mathfrak{m},a}$) and refer to these as pointed prime ideals (resp., pointed maximal ideals).
\end{enumerate}
\end{definition} 

\begin{proposition}\label{Notation}
For any prime ideal $\mathfrak{p}$ of $D$ and for any element $a$ of $E$, the following statements hold: 
\begin{enumerate}[$(1)$]   
\item $\intR(E,D)/\mathfrak{P}_{\mathfrak{p},a} \simeq D/\mathfrak{p}.$
\item  $\mathfrak{P}_{\mathfrak{p},a}$ is a prime ideal of $\intR(E,D)$ above $\mathfrak{p}$.
\item $\intR(E,D)_{\mathfrak{P}_{\mathfrak{p},a}} \cap K = D_{\mathfrak{p}}$.
\end{enumerate}
\end{proposition} 
	 
\begin{proof}
(1) It suffices to consider the map \[\begin{array}{cccc}
\psi_a:&\intR(E,D)&\rightarrow&D/\mathfrak{p}\\
&\varphi&\mapsto&\psi_a(\varphi)=\varphi(a) \mod(\mathfrak{p}).
\end{array}\]

(2) This is a consequence of the previous statement.

(3) The desired equality follows from statement (2) and Lemma \ref{Lolem}
\end{proof}

\begin{remark}
(1) We note that Proposition \ref{Notation}  justifies the notations $\mathfrak{P}_{\mathfrak{p},a}$ and  $\mathfrak{M}_{\mathfrak{m},a}$ used in the last statement of Definition \ref{Pointed}.

(2) Notably, as mentioned in \cite[page 259]{C97}, integer-valued rational functions do not behave well under localization, even in the Noetherian case. Recently, Liu, in \cite{BL22}, points out that the inclusion $S^{-1}\intR(D) \subseteq \intR(S^{-1} D)$ does not hold in general (see \cite[Example 1.22]{BL22}).
\end{remark}

\begin{proposition}
If $D=\cap_{\mathfrak{p}\in \mathcal{P}}D_{\mathfrak{p}},$ where $\mathcal{P}$ is a nonempty subset  of $\mathrm{Spec}(D)$, then $\intR(E,D)=\cap_{\mathfrak{p}\in \mathcal{P}} \intR(E,D_\mathfrak{p})= \cap_{a\in E}\cap_{\mathfrak{p}\in \mathcal{P}} \intR(E,D)_{\mathfrak{B}_{\mathfrak{p},a}}.$
\end{proposition}

\begin{proof}
By Proposition \ref{Comp}, we have  $\intR(E,D)\subseteq \intR(E,D_\mathfrak{p})$, for each $\mathfrak{p} \in \mathcal{P}$, and then $\intR(E,D) \subseteq \cap_{\mathfrak{p}\in \mathcal{P}} \intR(E,D_\mathfrak{p})$. For the other inclusion, we let $\varphi$ be an element of $\cap_{\mathfrak{p}\in \mathcal{P}} \intR(E,D_\mathfrak{p})$. For each $\mathfrak{p}\in \mathcal{P},$ we have $\varphi(E)\subseteq D_\mathfrak{p}$ and $\varphi \in K(X)$. Hence, $\varphi(E)\subseteq \cap_{\mathfrak{p}\in \mathcal{P}} D_\mathfrak{p}=D$, and thus $\cap_{\mathfrak{p}\in \mathcal{P}} \intR(E,D_\mathfrak{p})\subseteq \intR(E,D)$.

To prove the second equality, we only need to check the reverse inclusion. So, let $\psi \in \cap_{a\in E}\cap_{\mathfrak{p}\in \mathcal{P}} \intR(E,D)_{\mathfrak{P}_{\mathfrak{p},a}}$. Then, for each $\mathfrak{p}\in \mathcal{P}$ and each $a\in E$, $\psi \in \intR(E,D)_{\mathfrak{P}_{\mathfrak{p},a}},$ and hence there exists $\varphi \in \intR(E,D)\backslash \mathfrak{P}_{\mathfrak{p},a}$ such that $\varphi\psi\in \intR(E,D).$ Thus, $\varphi(a) \notin \mathfrak{p}$ and $\varphi(a) \psi(a) \in D$, and therefore $\psi(a) \in D_\mathfrak{p}$. Consequently, for each $\mathfrak{p}\in \mathcal{P}$, $\psi(E)\subseteq D_\mathfrak{p}$, which implies that $\psi(E)\subseteq \cap_{\mathfrak{p}\in \mathcal{P}} D_\mathfrak{p}=D$, that is, $\psi \in \intR(E,D)$.
\end{proof}

\begin{corollary}\label{reprLoc}
We always have the following:  
\begin{enumerate}[$(1)$]
\item $\intR(E,D)=\cap_{\mathfrak{p}\in\mathrm{Spec}(D)} \intR(E,D_\mathfrak{p})=\cap_{a\in E}\cap_{\mathfrak{p}\in\mathrm{Spec}(D)} \intR(E,D)_{\mathfrak{P}_{\mathfrak{p},a}}.$
\item $\intR(E,D)=\cap_{\mathfrak{m}\in\mathrm{Max}(D)} \intR(E,D_\mathfrak{m})=\cap_{a\in E}\cap_{\mathfrak{m}\in\mathrm{Max}(D)} \intR(E,D)_{\mathfrak{M}_{\mathfrak{m},a}}.$ 
\item If $(D,\mathfrak{m})$ is a local domain, then $\intR(E,D)=\cap_{a\in E} \intR(E,D)_{\mathfrak{M}_{\mathfrak{m},a}}.$
\end{enumerate}
\end{corollary}

\begin{theorem}\label{Loca}
If $D=\cap_{i}S_{i}^{-1}D,$ where $\{S_{i}\}_{i}$ is a family of multiplicative subsets of $D,$ then $ \intR(D)=\cap_{i}S_{i}^{-1}\intR(D).$
\end{theorem}
To prove this result we need the following two lemmas.

\begin{lemma}\label{ccv}
 Let $\{D_{i}\}_{i}$ be a family of integral domains with the same quotient field $K$ and $E$ be a subset of $K.$ Then, we have $\cap_{i}\intR(E,D_{i})\subseteq \intR\left(E,\cap_{i}D_{i}\right).$
\end{lemma}
\begin{proof}
Let $\varphi \in \cap_{i}\intR(E,D_{i})$. For each $i$, it follows that $\varphi(E)\subseteq D_{i}$ and $\varphi\in K(X)$. Then, $\varphi(E)\subseteq \cap_{i}D_{i}$, and thus, $\varphi \in \intR(E,\cap_{i}D_{i})$.
\end{proof}

\begin{lemma}\label{ll}
For any multiplicative subset $S$ of $D$, we have
    $S^{-1}\intR(D)\subseteq \intR(D,S^{-1}D).$
\end{lemma}
\begin{proof}
    Let $\varphi \in S^{-1}\intR(D)$. Then there exists an element $s$ of $S$ such that $s\varphi \in \intR(D)$, and hence $s\varphi\in K(X)$ and $s\varphi(d)\in D$, for all $d\in D$. Consequently, $\varphi\in K(X)$ and $\varphi(d)\in S^{-1}D$ for all $d\in D$, that is, $\varphi \in \intR(D,S^{-1}D)$. 
\end{proof}

\begin{proof}[Proof of Theorem \ref{Loca}]
 The direct inclusion is trivial. For the reverse inclusion, let $\varphi \in \cap_{i}S_{i}^{-1}\mathrm{Int}^{\mathrm{R}}(D)$. For each $i$, we have $\varphi \in S_{i}^{-1}\mathrm{Int}^{\mathrm{R}}(D)$. Then, it follows from Lemma \ref{ll} that  $\varphi \in \mathrm{Int}^{\mathrm{R}}(D,S_{i}^{-1}D)$ for all $i$. Hence, by Lemma \ref{ccv}, $\varphi \in \mathrm{Int}^{\mathrm{R}}(D,\cap_{i}S_{i}^{-1}D)=\mathrm{Int}^{\mathrm{R}}(D,D)=\mathrm{Int}^{\mathrm{R}}(D)$, and this completes the proof.
\end{proof}

\begin{corollary}
If $D=\cap_{\mathfrak{p}\in \mathcal{P}}D_{\mathfrak{p}},$ where $\mathcal{P}$ is a nonempty subset  of $\mathrm{Spec}(D),$ then $ \intR(D)=\cap_{\mathfrak{p}\in \mathcal{P}}\intR(D)_\mathfrak{p}$. In particular,  $\intR(D)=\cap_{\mathfrak{p}\in \mathrm{Spec}(D)}\intR(D)_\mathfrak{p}=\cap_{\mathfrak{m}\in \mathrm{Max}(D)}\intR(D)_\mathfrak{m}.$
\end{corollary}

{
\begin{remark}
It is known that the equality $\mathrm{Int}(D,S^{-1}D)=\mathrm{Int}(S^{-1}D)$ 
holds for any multiplicative subset $S$ of $D$ \cite[Corollary I.2.6]{C97}. Unfortunately, this is not the case for rings of integer-valued rational functions. In fact, if the equality $\mathrm{Int}^\mathrm{R}(D,S^{-1}D)=\mathrm{Int}^\mathrm{R}(S^{-1}D)$ holds, we deduce from Lemma \ref{ll} that  $S^{-1}\mathrm{Int}^\mathrm{R}(D)\subseteq\mathrm{Int}^\mathrm{R}(S^{-1}D)$, which is not always true as shown in \cite[Example 1.22]{BL22}.
\end{remark}}

\begin{proposition}
Let $\{S_{i}\}_{i}$ be a family of multiplicative subsets of $D$ such that $D=\cap_{i}S_{i}^{-1}D$. Then $\intR(E,D)=\cap_{i}\intR(E,S_{i}^{-1}D)$.
\end{proposition}

\begin{proof}
Since $D$ is contained in $S_{i}^{-1}D$ for each $i$, it follows that $\intR(E,D)\subseteq \intR(E,S_{i}^{-1}D)$ for each $i$, and hence $\intR(E,D)\subseteq \cap_{i}\intR(E,S_{i}^{-1}D)$. On the other side, from Lemma \ref{ccv}, we infer that $\cap_{i}\intR(E,S_{i}^{-1}D)\subseteq\intR(E,\cap_{i}S_{i}^{-1}D)=\intR(E,D)$. Thus, $\intR(E,D)=\cap_{i}\intR(E,S_{i}^{-1}D)$.
\end{proof}

In the remainder of this section, we will be interested in $v$- and $t$-ideals in the context of  $\intR(E,D)$. \\

The following result represents a special case of \cite[Lemma 1.9]{CLT00}.
\begin{proposition}\label{tstar}
Each $t$-ideal of $D$ is contained in a $t$-ideal of $\intR(E,D)$.
\end{proposition}

In what follows, we will show that the contraction of any $t$-ideal of $\intR(E,D)$ is a $t$-ideal of $D$. To achieve this, we first require the following useful lemma.
\begin{lemma}\label{lstar}
For any nonzero finitely generated ideal $J$ of $D,$ we have  $$(J\intR(E,D))_{v}\cap D=J_{v}.$$
\end{lemma}
\begin{proof}
Let $J=(a_1,...,a_n)$ be a nonzero finitely generated ideal of $D$. We first note that $J^{-1}=\cap_{1\leqslant i \leqslant n}\frac{1}{a_i}D$ and  $\left(J\intR(E,D)\right)^{-1}=\cap_{1\leqslant i \leqslant n}\frac{1}{a_i}\intR(E,D).$\\
If $x\in\left(J\intR(E,D)\right)_{v}\cap D$, then $x\left(J\intR(E,D)\right)^{-1}\subseteq \intR(E,D)$, and hence $xJ^{-1}\subseteq \intR(E,D)$ because $J^{-1}\subseteq \left(J\intR(E,D)\right)^{-1}$. Thus, $xJ^{-1}\subseteq \intR(E,D)\cap  K=D,$ that is, $x\in J_v$. For the other inclusion, if $x\in J_{v}$, then  $xJ^{-1}\subseteq D$. So, let $\varphi \in\left(J\intR(E,D)\right)^{-1}.$ Then, $\varphi\in\cap_{1\leqslant i \leqslant n}\frac{1}{a_i}\intR(E,D)$, and hence $\varphi(e)\in \cap_{1\leqslant i \leqslant n}\frac{1}{a_i}D=J^{-1}$, for all $e\in E$. Thus, $x\varphi(e) \in xJ^{-1} \subseteq D,$ for all $e\in E$, i.e., $x\varphi\in \intR(E,D)$. Therefore, $x\left(J\intR(E,D)\right)^{-1}\subseteq\intR(E,D)$  and so $x\in(J\intR(E,D))_{v}\cap D$. 
\end{proof}

\begin{proposition}\label{tcont}
For any nonzero ideal $I$ of $D$, we have $\left(I\intR(E,D)\right)_t\cap D=I_t.$ In particular, if $\mathfrak{A}$ is a $t$-ideal (resp., $t$-prime ideal) of $\intR(E,D)$ such that $\mathfrak{A}\cap D\neq(0)$, then $\mathfrak{A}\cap D$ is a $t$-ideal (resp., $t$-prime ideal) of $D$.
\end{proposition}
\begin{proof}
Let $I$ be a nonzero ideal of $D$. We have $\left(I\intR(E,D)\right)_t=\cup \mathfrak{J}_v,$ where $\mathfrak{I}$ ranges over the finitely generated ideals of $\intR(E,D)$ contained in $I\intR(E,D)$. As any of these $\mathfrak{J}$ is contained in a finitely generated ideal $J\intR(E,D)$ with $J$ is a finitely generated ideal of $D$ (contained in $I$), so we can write $\left(I\intR(E,D)\right)_t=\cup \left(J\intR(E,D)\right)_v,$ and then $\left(I\intR(E,D)\right)_t\cap D=\cup ((J\intR(E,D))_v \cap D)=\cup J_v$ (this last equality follows from Lemma \ref{lstar}). Therefore, $\left(I\intR(E,D)\right)_t\cap D=I_t.$ For the particular case, if $\mathfrak{A}$ is a $t$-ideal of $\intR(E,D)$ such that $\mathfrak{A}\cap D\neq(0)$, then the previous equality implies that $((\mathfrak{A}\cap D)\intR(E,D))_t\cap D=(\mathfrak{A}\cap D)_t$. Since $(\mathfrak{A}\cap D)\intR(E,D)\subseteq \mathfrak{A}$ and $\mathfrak{A}$ is a $t$-ideal, $((\mathfrak{A}\cap D)\intR(E,D))_t \subseteq \mathfrak{A}_t=\mathfrak{A}$, and hence $((\mathfrak{A}\cap D)\intR(E,D))_t\cap D\subseteq \mathfrak{A} \cap D$. Since $\mathfrak{A}\cap D$ is always contained in $(\mathfrak{A}\cap D)_t,$ we deduce the equality $(\mathfrak{A}\cap D)_t=\mathfrak{A}\cap D,$ and this means that $\mathfrak{A} \cap D$ is a $t$-ideal of $D$. The $t$-primeness follows from the fact that any contraction of a prime ideal of $\intR(E,D)$ is a prime ideal of $D$.
\end{proof}
As a corollary of this last result, we establish an analogue of Proposition \ref{tstar} for $t$-maximal ideals.
\begin{corollary}\label{Str}
Each $t$-maximal ideal of $D$ is the contraction of a $t$-maximal ideal of $\intR(E,D)$.
\end{corollary}

\begin{proof}
We first observe that if $\mathfrak{p}$ is a $t$-prime ideal of $D$, then $\mathfrak{p}\intR(E,D)$ is contained in a $t$-prime ideal of $\intR(E,D)$. Otherwise, we have $(\mathfrak{p}\intR(E,D))_t= \intR(E,D)$, and so it follows from Proposition \ref{tcont} that $\mathfrak{p}=\mathfrak{p}_t = (\mathfrak{p}\intR(E,D))_t \cap D = \intR(E,D)\cap D = D$, which is a contradiction. Thus, $\mathfrak{p}\intR(E,D)$ is contained in a $t$-maximal ideal $\mathfrak{M}$ of $\intR(E,D)$. Again, by Proposition \ref{tcont}, $\mathfrak{M}\cap D$ is a $t$-prime ideal of $D$. Consequently, if $\mathfrak{p}$ is assumed to be $t$-maximal, then $\mathfrak{M}\cap D = \mathfrak{p}$, and this completes the proof.   
\end{proof}

\begin{corollary}
The extension $D\subseteq\intR(E,D)$ is $t$-compatible, that is, $I_t\subseteq\left(I\intR(E,D)\right)_t,$ for any nonzero fractional ideal $I$ of $D$.
\end{corollary}

\begin{proposition}\label{tvId}
For any nonzero fractional ideal $I$ of $D$, we have:
\begin{enumerate}[$(1)$]
\item  $\left(I \intR(E,D)\right)^{-1}= \intR(E,I^{-1})=\left( \intR(E,I)\right)^{-1},$
\item  $\left(I \intR(E,D)\right)_{v}=\left(I_{v} \intR(E,D)\right)_{v}=\left( \intR(E,I)\right)_{v}=\intR(E,I_{v}),$
\item $\left(I \intR(E,D)\right)_{t}=\left(I_{t} \intR(E,D)\right)_t$.
\end{enumerate}
\end{proposition}
\begin{proof}
(1) It is clear that $\intR(E,I)\intR(E,J)\subseteq \intR(E,IJ)$, for any two fractional ideals  $I$ and $J$ of $D$. Then, we infer that $\intR(E,I^{-1}) \subseteq \left(\intR(E,I)\right)^{-1} \subseteq \left(I \intR(E,D)\right)^{-1}$ (the second inclusion follows from the fact that $I \intR(E,D) \subseteq \intR(E,I)$). For the reverse inclusion, let $\varphi \in \left(I \intR(E,D)\right)^{-1}$. For  each element $x$ of  $I$, we have $x\varphi\in \intR(E,D),$ and then  $x\varphi(e)\in D$, for each $e \in E.$ Therefore, $\varphi \in \intR(E,I^{-1})$.

(2) From statement (1), we have $\left(I \intR(E,D)\right)_{v}=\left( \intR(E,I)\right)_{v}=\intR(E,I_{v}).$ So, the statement follows from the inclusions  $I \intR(E,D) \subseteq I_{v} \intR(E,D) \subseteq \intR(E,I_{v})$.

(3) This follows from the fact that $\left(I \intR(E,D)\right)_{v}=\left(I_{v} \intR(E,D)\right)_{v},$ for all nonzero fractional ideals $I$ of $D$, and by using \cite[Proposition 2.6]{Z00}.
\end{proof}
\begin{remark}
Notice that $\intR(E,I_{v})$ is a $v$-ideal of $\intR(E,D)$ because it follows from the previous result that $\left( \intR(E,I_v)\right)_{v}=\intR(E,(I_{v})_v)=\intR(E,I_{v}).$
\end{remark}

\begin{corollary}\label{i-1}
For any nonzero fractional ideal $I$ of $D$, we have $$(I\intR(E,D))^{-1}\cap D\subseteq I^{-1}.$$
\end{corollary}
\begin{proof}
This follows from the first statement of the above proposition.
\end{proof}

An ideal $J$ of $D$ is called a {\it Glaz-Vasconcelos ideal} (in short, GV-ideal), if $J$ is finitely generated and $J^{-1}= D$. The set of all Glaz-Vasconcelos ideals of $D$ is denoted by $\mathrm{GV}(D)$.

\begin{proposition}\label{G-V}
Let $J$ be a nonzero fractional ideal of $D$. If $J\in\mathrm{GV}(D)$ then $J\intR(E,D)\in\mathrm{GV}(\intR(E,D)),$ and the converse holds when $J$ is finitely generated.
\end{proposition}

\begin{proof}
Assume that $J\in\mathrm{GV}(D)$. By definition, $J$ is  finitely generated and $J^{-1}= D,$ and then $J\intR(E,D)$ is also finitely generated. Hence it follows from Proposition \ref{tvId} that $\left(J \intR(E,D)\right)^{-1}= \intR(E,J^{-1})=\intR(E,D),$ and thus $J\intR(E,D)\in\mathrm{GV}(\intR(E,D)).$ For the converse, assume that $J$ is finitely generated such that $J\intR(E,D)\in\mathrm{GV}(\intR(E,D)).$ First, we deduce from Corollary \ref{i-1} that $D\subseteq J^{-1}$ because $\left(J \intR(E,D)\right)^{-1}= \intR(E,D).$ Now, let $x$ be an element of $J^{-1}$. We have $xJ\subseteq D$, and then $xJ\intR(E,D)\subseteq \intR(E,D)$. Hence, $x\in (J\intR(E,D))^{-1}=\intR(E,D)$ (this equality is due to the fact that $J\intR(E,D)\in\mathrm{GV}(\intR(E,D))$, and thus $x\in \intR(E,D)\cap K=D$. Therefore, $J^{-1}\subseteq D$, and so $J^{-1}=D$.  Combining this last equality with the fact that $J$ is finitely generated, we infer that $J\in \mathrm{GV}(D).$
\end{proof}

For any subset $\mathfrak{I}$ of $K(X)$ and for any element $a$ of $K$, we set $\mathfrak{I}(a)=\{\varphi(a);\;\varphi\in\mathfrak{I}\}$. Clearly, $\intR(E,D)(a)=D,$ for any element $a$ of $E$.
\begin{theorem}\label{Moriw}
For any  nonzero ideal $I$ of $D$, we have $I_{w}\subseteq(I\intR(E,D))_{w}\cap D,$ and the equality holds if $D$ is a Mori domain.
\end{theorem}
\begin{proof}
Let $I$ be a nonzero ideal of $D$ and $x$ an element of $I_{w}$. We have $xJ\subseteq I,$ for some $J\in  \mathrm{GV}(D)$, and then $xJ\intR(E,D)\subseteq I\intR(E,D)$. Hence, by Proposition \ref{G-V}, $J\intR(E,D)\in\mathrm{GV}(\intR(E,D))$, and thus $x\in (I\intR(E,D))_{w}\cap D$. Therefore, $I_{w}\subseteq (I\intR(E,D))_{w}\cap D$.

Now, assume that $D$ is Mori and let $x\in (I\intR(E,D))_{w}\cap D$. We have $xJ\subseteq I\intR(E,D)$ for some $J\in \mathrm{GV}(\intR(E,D))$, and then $J(e)\subseteq D$ and $xJ(e)\subseteq I,$ for all $e\in E$. So, let us consider $J'$  the ideal of $D$ generated by the union $\cup_{e\in E}J(e)$. Since $D$ is a Mori domain, it follows from \cite[Theorem 2.1\rm(2)]{B00} that there exists a finite subset of $\cup_{e\in E}J(e)$, namely $\{a_1,...,a_n\}$, such that $J'^{-1}=(a_1,...,a_n)^{-1}.$ Hence, from the fact that $xJ'\subseteq I$ it follows that $x(a_1,...,a_n)\subseteq I,$ and so we need only to show that $(a_1,...,a_n)\in\mathrm{GV}(D),$ i.e., $(a_1,...,a_n)^{-1}=D.$ To do this, let $y\in(a_1,...,a_n)^{-1}$. Since $J'^{-1}=(a_1,...,a_n)^{-1},$ $yJ'\subseteq D$ and then $y\varphi(e)\in D,$ for each $\varphi\in J$ and each $e\in E.$ Hence, $y\varphi\in\intR(E,D),$ for each $\varphi\in J,$ which means that $yJ\subseteq \intR(E,D).$ Thus, $y\in J^{-1}=\intR(E,D)$ because $J\in \mathrm{GV}(\intR(E,D)),$ and therefore, $y\in\intR(E,D)\cap K=D.$ Consequently, $(a_1,...,a_n)\in\mathrm{GV}(D),$ and this proves the desired equality.
\end{proof}

We conclude this section by asking the following question:
\begin{question} 
Does the $w$-analogue of the third statement of Proposition \ref{tvId} hold true in general, or at least for Mori domains?
\end{question}

\section{The transfer of some properties in  $\intR(E,D)$}\label{sec4}
In this section, we aim to investigate the transfer of some ring-theoretic properties to rings of integer-valued rational functions.

\medskip

We start with a complete characterization of when $\intR(E,D)$ is integrally closed, which proposed in \cite{C97} as an exercise.	 
\begin{proposition}[{\cite[Exercise X.2]{C97}}]\label{11}
The integral domain $\intR(E,D)$ is integrally closed if and only if so is $D$.
\end{proposition}
	 
\begin{proof}
Suppose that $D$ is integrally closed and let $\psi \in K(X)$ be integral over $\intR(E,D)$. We have: 
	 	$$\psi^{n}+\varphi_{n-1}\psi^{n-1}+\dots +\varphi_{1}\psi +\varphi_{0}=0,$$
	 	where $\varphi_{i}\in\intR(E,D)$ for each $i$. Then, for each $e\in E$, we have : 
	 	$$\psi^{n}(e)+\varphi_{n-1}(e)\psi^{n-1}(e)+\dots +\varphi_{1}(e)\psi(e) +\varphi_{0}(e)=0,$$
and hence $\psi(e)$ is integral over $D$. Therefore $\psi(e)\in D$, for each $e\in E$, that is, $\psi \in \intR(E,D)$.  For the converse, if $\intR(E,D)$ is integrally closed  then so is $D$ because $D=\intR(E,D)\cap K$ (an intersection of two integrally closed domains). 
\end{proof}
	 
\begin{remark}
This last proposition implies that the integral closure of $\intR(E,D)$ is contained in $\intR(E,D^{\prime}),$ where $D^{\prime}$ denotes the integral closure of $D$.
\end{remark}	 

An integral domain $D$ with quotient field $K$ is said to be \textit{seminormal} if, for each $\alpha\in K$, whenever $\alpha^2,\alpha^3\in D$, then $\alpha\in D$. Example of seminormal domains are Pr\"ufer domains. It is noteworthy that $\mathrm{Int}(D)$ is a seminormal domain  if and only if so is $D$ (cf. \cite[Proposition 7.2\rm(3)]{AAZ91}). We next provide an analogue of this result for $\intR(E,D)$.
\begin{proposition}
The integral domain $\intR(E,D)$ is seminormal if and only if so is $D$.
\end{proposition}
\begin{proof}
Assume that $\intR(E,D)$ is seminormal and let $\alpha \in K$ such that $\alpha^{2},\alpha^{3}\in D$. Since $D\subseteq \intR(E,D)$, we have $\alpha^{2},\alpha^{3}\in \intR(E,D)$, and then by seminormality of $\intR(E,D)$, $\alpha \in \intR(E,D)$. Hence, $\alpha \in D$ and thus $D$ is seminormal. Conversely, assume that  $D$ is seminormal and let $\varphi \in K(X)$ with $\varphi^{2},\varphi^{3}\in \intR(E,D)$. We have $\varphi^{2},\varphi^{3}\in \intR(E,D)$ implies that $\varphi^{2}(e),\varphi^{3}(e)\in D,$ for all $e\in E$, and hence $\varphi(e)\in D$ by seminormality of $D$, for all $e\in E$. Thus,  $\varphi \in \intR(E,D),$ and the proof is done. 
\end{proof}

In the following, we will provide a necessary condition on $D$ for which the integral domain $\intR(E,D)$ is either Krull, Mori, almost Krull, or locally Mori.
\begin{proposition}\label{transf}
Let $(\mathcal{P})$ denotes one of the following properties for domains: Krull, Mori, almost Krull or locally Mori. If $\intR(E,D)$ has the property $(\mathcal{P})$ then $D$ has the same property.
\end{proposition}
\begin{proof}
Since $D=\intR(E,D)\cap K$ and the intersection of two Krull (resp., Mori) domains is a Krull (resp., Mori) domain, if  $\intR(E,D)$ is Krull (resp., Mori) then  $D$ is also Krull (resp., Mori). Now, assume that $\intR(E,D)$ is an almost Krull domain, and let $\mathfrak{m}$ be a maximal ideal of $D$ and $a$ an element of $E$. It follows from Proposition \ref{Notation} that $\mathfrak{M}_{\mathfrak{m},a}$ is a maximal ideal of $\intR(E,D)$ and then $\intR(E,D)_{\mathfrak{M}_{\mathfrak{m},a}}$ is a Krull domain. Again, by Proposition \ref{Notation}, $D_\mathfrak{m}=\intR(E,D)_{\mathfrak{M}_{\mathfrak{m},a}}\cap K,$ and so $D_\mathfrak{m}$ is a Krull domain as an intersection of two Krull domains. Therefore, $D$ is an almost Krull domain. Lastly, the case of locally Mori is similar to that of almost Krull.
\end{proof}
For strong Mori domains, we have the following:
\begin{proposition}
If $\intR(E,D)$ is a strong Mori domain then so is $D$.
\end{proposition}
\begin{proof}
Assume that $\intR(E,D)$ is a strong Mori domain. By \cite[Corollary 3.4]{CHK12}, we need to show that every $w$-prime ideal of $D$ is of $w$-finite type.  So, let $\mathfrak{p}$ be a $w$-prime ideal of $D$. Since $\intR(E,D)$ is a strong Mori domain, $(\mathfrak{p}\intR(E,D))_{w}$ is of $w$-finite type, and then $(\mathfrak{p}\intR(E,D))_{w} = (I\intR(E,D))_{w},$ for some finitely generated subideal $I$ of $\mathfrak{p}$. By Proposition \ref{transf}, $D$ is a Mori domain because any strong Mori domain is Mori, and then, by Theorem \ref{Moriw}, $\mathfrak{p}_{w}=(\mathfrak{p}\intR(E,D))_{w}\cap D=(I\intR(E,D))_{w}\cap D=I_{w}$.  Thus, $\mathfrak{p} = I_{w}$ is of $w$-finite type, and therefore $D$ is strong Mori.
\end{proof}
\begin{remark}
It is important to note that the converse of the previous two propositions does not hold in general. For instance, while the ring of integers $\mathbb{Z}$ is locally Mori (as it is a PID), $\intR(\mathbb{Z})$ is not locally Mori, and consequently it is neither (almost) Krull nor (strong) Mori. This is because any Pr\"ufer locally Mori domain must be one-dimensional, whereas $\intR(\mathbb{Z})=\mathrm{Int}(\mathbb{Z})$ is known to be a two-dimensional Pr\"ufer domain.
\end{remark}
We now present the analogue of \cite[Proposition 1.1 and Corollary 1.2]{TT22} for $\intR(E,D)$.
\begin{proposition}\label{tFin}
If $\intR(E,D)$ has finite $($resp., $t$-finite$)$ character then so is $D$.
\end{proposition}
\begin{proof}
By contraposition, assume that $D$ has not finite character. Then there is exists a nonzero element $x$ of $D$ that is contained in infinitely many maximal ideals of $D$, and hence it is contained in infinitely many maximal ideals of $\intR(D)$ (As a matter of fact: for any maximal ideal $\mathfrak{m}$ of $D$, we can consider the maximal pointed ideal $\mathfrak{M}_{\mathfrak{m},a}$ for some element $a$ of $E$). Thus, $\intR(E,D)$ has not finite character. Now, suppose that $D$ has not  $t$-finite character. Then there is a nonzero element of $D$ which is contained in infinitely many $t$-maximal ideals of $D$. Then, by Corollary \ref{Str}, each $t$-maximal ideal of $D$ is the contraction of a $t$-maximal ideal of $\intR(E,D)$, and so $\intR(E,D)$ has not $t$-finite character. 
\end{proof}
From the previous result, we derive immediately the following:
\begin{corollary}
If $\intR(D)$ has finite $($resp., $t$-finite$)$ character then so is $D$.
\end{corollary}

 We  recall the following definition of locally finite intersection of integral domains. Let $\{D_\alpha\}_{\alpha\in\Lambda}$ be a family of integral domains having  the same quotient field. The intersection $\cap_{\alpha\in\Lambda}D_\alpha =: D$ is said to be \textit{locally finite} if every nonzero element of $D$ is a unit in $D_\alpha$ for all but finitely many $\alpha\in\Lambda$. In particular, if each $D_\alpha$ is local with maximal ideal $\mathfrak{m}_\alpha$, the above intersection is locally finite if and only if each nonzero element of $D$ belongs to only finitely many ideals $\mathfrak{m}_\alpha$.
\begin{proposition}\label{f}
If $\intR(E,D)$ is a locally finite intersection of a family of its localizations then so is $D$.
\end{proposition}
\begin{proof}
Suppose that $\intR(E,D) = \cap_{\mathfrak{P}\in \mathcal{P}}\intR(E,D)_{\mathfrak{P}}$, where $\mathcal{P} \subseteq  \mathrm{Spec}(\intR(E,D))$ such that the intersection $\cap_{\mathfrak{P}\in \mathcal{P}}\intR(E,D)_{\mathfrak{P}}$ is locally finite. We set $\mathcal{P}^{\prime} := \lbrace \mathfrak{P} \cap D; \mathfrak{P} \in \mathcal{P}\rbrace$. By Lemma \ref{Lolem}, we have  $D_{\mathfrak{P}\cap D}=\intR(E,D)_{\mathfrak{P}} \cap K$, and then $\cap_{\mathfrak{p}\in \mathcal{P}^{\prime}} D_{\mathfrak{p}}=\cap_{\mathfrak{P}\in \mathcal{P}}(\intR(E,D)_{\mathfrak{P}} \cap K)=\intR(E,D) \cap K=D.$
Thus, $D = \cap_{\mathfrak{p}\in \mathcal{P}^{\prime}} D_{\mathfrak{p}}$ and this intersection is locally finite since the intersection $\cap_{\mathfrak{P}\in \mathcal{P}}\intR(E,D)_{\mathfrak{P}}$ is locally finite.
\end{proof}
\begin{remark}
The converse of the previous proposition is not true in general. For instance, $\mathbb{Z}$ has finite character but $\intR(\mathbb{Z})$ does not have a locally finite intersection of a family of its localizations. Otherwise, $\intR(\mathbb{Z})=\mathrm{Int}(\mathbb{Z})$ is a Krull-type domain because any Pr\"ufer domain with locally finite intersection is Krull-type, and then it follows from \cite[Theorem 2.30]{TT23} that $\intR(\mathbb{Z})=\mathrm{Int}(\mathbb{Z})=\mathbb{Z}[X],$ which is not the case.
\end{remark}
We next  show that the property of being a P$v$MD transfers from $\intR(E,D)$ to $D,$ but not conversely.
\begin{proposition}\label{pvmd}
If $\intR(E,D)$ is a P$v$MD then so is $D$.
\end{proposition}
\begin{proof}
Assume that $\intR(E,D)$ is a P$v$MD and let $\mathfrak{m}$ be a $t$-maximal ideal of $D$.  From Corollary \ref{Str}, we deduce that there exists a $t$-maximal ideal $\mathfrak{M}$ of $\intR(E,D)$ such that $\mathfrak{M} \cap D = \mathfrak{m}$. Thus, $\mathfrak{m}\intR(E,D) \subseteq \mathfrak{M}$. Since $\intR(E,D)$ is a P$v$MD, $\intR(E,D)_{\mathfrak{M}}$ is a valuation domain and then so is $ D_\mathfrak{m}$  because $D_\mathfrak{m}=\intR(E,D)_{\mathfrak{M}} \cap K$ by Lemma \ref{Lolem}, and therefore $D$ is a P$v$MD.
\end{proof}
\begin{remark}
It is worth noting that the necessary condition stated in Proposition \ref{pvmd} is not sufficient in general. As a matter of fact, in \cite[Remark 2.4.4]{BL23}, the author showed that if $V$ is a valuation domain (and hence it is is a P$v$MD)  with non-principal maximal ideal and algebraically closed residue field then $\intR(V)$ is an essential domain that is not a P$v$MD. Consequently, the characterization of when $\intR(E,D)$, or at least $\intR(D)$, is a P$v$MD remains an open problem.
\end{remark}

We recall that  \textit{Krull-type domains} are integral domains $D$ for which $D=\cap_{\mathfrak{p}\in\mathcal{P}}D_\mathfrak{p},$ where $\mathcal{P}\subseteq\mathrm{Spec}(D),$ $D_\mathfrak{p}$ is a valuation domain for each $\mathfrak{p}\in\mathcal{P}$ and the intersection is locally finite. It is well-known that P$v$MDs of $t$-finite character coincide with Krull-type domains. From this last fact, it follows from Propositions \ref{tFin}  and \ref{pvmd} the following:
\begin{corollary}\label{KrTy}
If $\intR(E,D)$ is a Krull-type domain then so is $D$.
\end{corollary}

We now prove that $\intR(E,D)$ being a $t$-almost Dedekind domain forces $D$ to also be a $t$-almost Dedekind domain.
\begin{proposition}\label{t-AD}
If $\intR(E,D)$ is a $t$-almost Dedekind domain then so is $D$.
\end{proposition}
\begin{proof}
Assume that $\intR(E,D)$ is a $t$-almost Dedekind domain and let $\mathfrak{m}$ be a $t$-maximal ideal of $D$.  From Corollary \ref{Str}, we deduce that there exists a $t$-maximal ideal $\mathfrak{M}$ of $\intR(E,D)$ such that $\mathfrak{M} \cap D = \mathfrak{m}$. Then, as in the proof of Proposition \ref{pvmd}, $D_\mathfrak{m}=\intR(E,D)_{\mathfrak{M}} \cap K$ is a DVR as an intersection of two DVRs, and hence $D$ is $t$-almost Dedekind.
\end{proof}

\begin{proposition}\label{t-Dim1}
If $t$-$\dim(\intR(E,D))=1,$ then $D$ is either a field or of $t$-dimension 1.
\end{proposition}
\begin{proof}
Suppose that $t$-$\dim(\intR(E,D))=1$ and $D$ is not a field, and let $\mathfrak{m}$ be a $t$-maximal ideal of $D$. Then there is $t$-prime ideal $\mathfrak{p}$ of $D$ contained in $\mathfrak{m}$, and from the observation made in the beginning of the proof of Corollary \ref{Str} we deduce that there is a $t$-prime ideal $\mathfrak{P}$ of $\intR(E,D)$ containing $\mathfrak{p}$. Also, by Corollary  \ref{Str}, $\mathfrak{m}$ is the contraction of a $t$-maximal ideal $\mathfrak{M}$ of  $\intR(E,D)$. Thus, from the inclusion $(0)\subset\mathfrak{P}\cap \mathfrak{M}\subseteq \mathfrak{M}$ and using the fact that  $\mathfrak{P}\cap \mathfrak{M}$ is a $t$-prime ideal of  $\intR(E,D)$ and $t$-$\dim(\intR(E,D)=1$, we have $\mathfrak{P}\cap \mathfrak{M}=\mathfrak{M}$. Therefore, $\mathfrak{M}\subseteq \mathfrak{P}$, which implies that $\mathfrak{M}=\mathfrak{P}$, and thus $\mathfrak{m}=\mathfrak{p}$, that is, $t$-$\dim(D)=1$. 
\end{proof}
\begin{remark}
We note that the converse of the last two propositions is not true in general. Indeed, the integral domain $\mathbb{Z}$ is $t$-almost Dedekind (and then it has $t$-dimension one). However, the $t$-dimension of $\intR(\mathbb{Z})$ is equal to two because $\intR(\mathbb{Z})=\mathrm{Int}(\mathbb{Z})$ and $\mathrm{Int}(\mathbb{Z})$ is known to be a two-dimensional Pr\"ufer domain, and also $\intR(\mathbb{Z})$ is not $t$-almost Dedekind.
\end{remark}

An integral domain $D$ is \textit{weakly-Krull} if $D=\cap_{\mathfrak{p}\in X^1(D)}D_\mathfrak{p}$ and  this intersection is locally finite.  Clearly,  Krull domains are weakly-Krull. An integral domain $D$ is said to be a \textit{generalized Krull domain} (in the sense of Gilmer \cite[Section 43]{Gi92}), if $D=\cap_{\mathfrak{p}\in X^1(D)} D_\mathfrak{p}$, where the intersection is locally finite and each $D_\mathfrak{p}$ is a valuation domain. For instance, Krull domains are generalized Krull, and generalized Krull domains are a subclass of both Krull-type domains and weakly-Krull domains.

\begin{corollary}
If $\intR(E,D)$ is a weakly-Krull $($resp., generalized Krull$)$ domain then so is $D$.
\end{corollary}
\begin{proof}
By combining the fact that weakly-Krull domains are exactly the one $t$-dimensional domains with $t$-finite character,  Propositions \ref{tFin} and \ref{t-Dim1}, we deduce the transfer of the weakly-Krull property from $\intR(E,D)$ to $D$. The generalized Krull property follows from the fact that a generalized Krull domain is exactly a Krull-type domain of $t$-dimension one, which can be obtained by using  Corollary \ref{KrTy} and Proposition \ref{t-Dim1}.
\end{proof}

Given a nonempty subset $\mathcal{P}$ of $\mathrm{Spec}(D)$, we say that $D$ is an \textit{essential domain} with defining family $\mathcal{P}$ if $D=\cap_{\mathfrak{p}\in\mathcal{P}}D_\mathfrak{p}$ and $D_\mathfrak{p}$ is a valuation domain for each $\mathfrak{p}\in\mathcal{P}$. It is clear that almost Krull domains and P$v$MDs are essential, and essential domains are integrally closed.

We next show that the essentiality transfers from $\intR(E,D)$ to $D$.
\begin{proposition}\label{essential}
If $\intR(E,D)$ is an essential domain then so is $D$.
\end{proposition}
\begin{proof}
Suppose that $\intR(E,D) = \cap_{\mathfrak{P}\in \mathcal{P}}\intR(E,D)_{\mathfrak{P}}$ is an essential domain with defining family $\mathcal{P}$ and setting  $\mathcal{P}^{\prime} := \lbrace \mathfrak{P} \cap D; \mathfrak{P} \in \mathcal{P}\rbrace$. For any $\mathfrak{P}\in\mathcal{P},$ we have $\intR(E,D)_{\mathfrak{P}} $ is a valuation domain, and hence, by Lemma \ref{Lolem}, $D_{\mathfrak{P}\cap D}=\intR(E,D)_{\mathfrak{P}} \cap K$ is also a valuation domain. Moreover, as in the proof of Proposition \ref{f}, we have  $D = \cap_{\mathfrak{p}\in \mathcal{P}^{\prime}} D_{\mathfrak{p}},$ and so $D$ is an essential domain with defining family $\mathcal{P}^{\prime}.$
\end{proof}

The following result gives a sufficient condition on valuation domains $V$ for which $\intR(E,V)$ is an essential domain.
\begin{proposition}\label{PrEssential}
Let $(V,\mathfrak{m})$ be a valuation domain. If the value group of $V$ is not divisible, then $\intR(E,V)$ is an essential domain.  
\end{proposition}
\begin{proof}
Assume that the value group of $V$ is not divisible. By Corollary \ref{reprLoc}(3) and \cite[Proposition 2.35]{BL22}, we have $\intR(E,V)=\cap_{a\in E} \intR(E,V)_{\mathfrak{M}_{\mathfrak{m},a}}$ and each $\intR(E,V)_{\mathfrak{M}_{\mathfrak{m},a}}$ is a valuation domain. Therefore, $\intR(E,V)$ is an essential domain. 
\end{proof}
\begin{corollary}\label{CorKT}
Let $V$ be a valuation domain whose value group is not divisible. If $E$ is finite, then $\intR(E,V)$ is a Krull-type domain, and hence it is a P$v$MD.  
\end{corollary}
\begin{proof}
Under the given assumptions, the integral domain $\intR(E,V)$  can be written as a finite intersection of some family of its valuation overrings.
\end{proof}
{
Since the ring of integers $\mathbb{Z}$ is not divisible, we have:
\begin{example}
Let $V$ be a valuation domain with value group $\mathbb{Z}$ (for examples of such domains, we can take any DVR or refer to \cite[Examples 6.7.1 and 6.7.3]{HS06}, \cite[Example 2.7]{S90}, or \cite[Examples 8(ii) and 9]{V06}), and let $E$ be a subset of the quotient field of $V$. It follows from Proposition \ref{PrEssential} that $\intR(E,V)$ is an essential domain, and it is Krull-type if $E$ is finite as asserted Corollary \ref{CorKT}.
\end{example}}
We do not know whether the converse of each of the last two propositions is true or not.

\section{The module structure of $\intR(E,D)$ and  $\mathrm{Int}_B(E,D)$ }\label{sec5}
	  
In this section, we will investigate the properties of   $\intR(E,D)$ as a module over $D$ and also as an overring of $D[X]$ when $E$ is a subset of $D$. We first start by stating the following remark:
	 
\begin{remark}\label{123}
We have the following :
\begin{enumerate}[$(1)$]
\item {The ring $\intR(E,D)$ also has the structure of a module over $D$.}
\item The $D$-module $\intR(E,D)$ is torsion-free.
\item $\intR(E,D)\cap K=D.$ 
\end{enumerate}
\end{remark}

Based on the last statement of the previous remark and by virtue of \cite[Remark 3.4]{HOR16}, we derive the following:
\begin{proposition}\label{FFPr}
If $D$ is a Pr\"ufer domain, then $\intR(E,D)$ is faithfully flat as a $D$-module.
\end{proposition}
 
As an illustrative example, consider the following:
\begin{example}
The integral domain $D = \mathbb{Z} + T\mathbb{Q}[T]$, where $T$ is an indeterminate over $\mathbb{Q}$, is known to be Pr\"ufer. By Proposition~\ref{FFPr}, $\intR(E,D)$ is a faithfully flat $D$-module for any subset $E$ of $\mathbb{Q}(T)$.
\end{example} 
 
We next provide the $w$-analogue of the previous result. Before doing this, let us recall two necessary concepts.

Following \cite{KW14}, a module $M$ over $D$ is said to be $w$-\textit{flat}  if, for every short exact sequence $0\rightarrow A\rightarrow B\rightarrow C\rightarrow0$ of $D$-modules, $0\rightarrow (M\otimes_DA)_w\rightarrow (M\otimes_DB)_w\rightarrow (M\otimes_DC)_w\rightarrow0$ is also exact. It is noteworthy that flatness implies $w$-flatness.  A module $M$ over $D$ is said to be $w$-\textit{faithfully flat} if it is $w$-flat and $(M/\mathfrak{p}M)_w\neq0$ for all $\mathfrak{p}\in w$-$\mathrm{Max}(D)$. 	 
\begin{proposition}
For any P$v$MD $D$, the $D$-module $\intR(E,D)$ is $w$-faithfully flat.
\end{proposition}	 
\begin{proof}
Using \cite[Proposition 2.5]{HK13}), we only need to show that $\intR(E,D)_\mathfrak{m}$ is faithfully flat $D_\mathfrak{m}$-module for all $w$-maximal ideals $\mathfrak{m}$ of $D$. To do this, let $\mathfrak{m}$ be a $w$-maximal ideal of $D$. As $D$ is a P$v$MD, $D_\mathfrak{m}$ is a valuation domain and then it follows from \cite[Remark 3.4]{HOR16} that $\intR(E,D)_\mathfrak{m}$ is a faithfully flat $D_\mathfrak{m}$-module because $D_\mathfrak{m}=(\intR(E,D)\cap K)_\mathfrak{m}=\intR(E,D)_\mathfrak{m}\cap K$. Thus, $\intR(E,D)$ is $w$-faithfully flat $D$-module, as desired.
\end{proof}

This last result and its proof represent a particular case of a result established in \cite{COT23}. For the sake of completeness, we have included the previous proof.

\medskip

Our next theorem shows that the concepts of flatness and faithful flatness coincide for the two $D$-modules $\intR(E,D)$ and   $\mathrm{Int}_B(E,D)$.

\begin{theorem}\label{l}
The $D$-module $\intR(E,D)$ $($resp.,  $\mathrm{Int}_B(E,D))$ is flat if and only if it is faithfully flat.
\end{theorem}	 
	 
\begin{proof}
Assume that $\intR(E,D)$ is flat. To prove that $\intR(E,D)$  is faithfully flat $D$-module, it suffices to check that, for all maximal ideals $\mathfrak{m}$ of $D$, we have $\mathfrak{m}\intR(E,D)\neq \intR(E,D)$ (cf. \cite[Theorem 2, pages 25 and 26]{HM22}). By way of contradiction, suppose there exists a maximal ideal $\mathfrak{m}$ of $D$ such that $\mathfrak{m}\intR(E,D)= \intR(E,D)$. This implies   $1=m_{1}\varphi_{1}+\dots +m_{n}\varphi_{n}$ for some $m_{1},\dots ,m_{n} \in \mathfrak{m}$ and $\varphi_{1},\dots ,\varphi_{n} \in \intR(E,D)$.  So, necessarily  at least one of  $\varphi_{i},\dots ,\varphi_{n}$ must be nonzero (otherwise, we would have $1=0$, a contradiction). Then, $1=m_{1}\varphi_{1}(e)+\dots +m_{n}\varphi_{n}(e)\in\mathfrak{m}$, for all $e\in E$ (because $\varphi_{1},\dots ,\varphi_{n} \in \intR(E,D)$). Hence, $1\in \mathfrak{m}$, which is a contradiction. Consequently,  $\intR(E,D)$  is faithfully flat $D$-module. The converse  is trivial. 
	 		
The proof for the case of $\mathrm{Int}_B(E,D)$ is similar to that of $\intR(E,D)$.
\end{proof}
	 
\begin{proposition}
For any nonempty subset $E$ of $D$, the following statements are equivalent:
\begin{enumerate}[$(1)$]
\item $\intR(E,D)$ is faithfully flat over $D[X]$;
\item $\intR(E,D)$ is $w$-faithfully flat over $D[X]$;
\item  $\intR(E,D)=D[X]$.
\end{enumerate}
\end{proposition}
\begin{proof}
The implications (1) $\Rightarrow$ (2) and (3) $\Rightarrow$ (1) are straightforward. The implication (2) $\Rightarrow$ (3) follows from \cite[Corollary 2.6]{HK13} since $D[X]$ and $\intR(E,D)$ share the same quotient field.
\end{proof}
	 
As an immediate consequence of the previous result, we have: 
\begin{corollary}\label{CorFF1}
If $\mathrm{Int}(D)$ is not trivial $($that is, $\mathrm{Int}(D)\neq D[X])$, then $\intR(E,D)$ is not $(w$-$)$faithfully flat over $D[X]$, for any nonempty subset $E$ of $D$. 
\end{corollary}

This last corollary allows us to provide the following example.
\begin{example}
Let $E$ be a nonempty subset of $\mathbb{Z}$.  Since $\mathrm{Int}(\mathbb{Z})\neq \mathbb{Z}[X]$, it follows from Corollary \ref{CorFF1} that  $\intR(E,\mathbb{Z})$ is not ($w$-)faithfully flat over $\mathbb{Z}[X].$ On the other hand, by Proposition \ref{FFPr},   $\intR(E,\mathbb{Z})$ is faithfully flat as a $\mathbb{Z}$-module (because  $\mathbb{Z}$ is  a PID). Remarkably, when $E$ is infinite, $\intR(E,\mathbb{Z})$ coincides with $\mathrm{Int}(E,\mathbb{Z})$ as asserted Proposition X.1.1 of \cite{C97}, and then $\intR(E,\mathbb{Z})$ is not only faithfully flat as a $\mathbb{Z}$-module but also free with a regular basis by \cite[Corollary II.1.6]{C97}.
\end{example} 

All established results regarding the flatness of $\mathrm{Int}(D)$ as a $D[X]$-module assert that if the flatness of $\mathrm{Int}(D)$ holds over $D[X],$  then $\mathrm{Int}(D)$ is necessarily trivial. However, the following proposition shows that it is possible to have  $\intR(D)$ flat over $D[X]$ while $\intR(D)\neq D[X]$. Furthermore, it shows that $\intR(D)$ can be faithfully flat over $D$ without being faithfully flat over $D[X]$.

\begin{proposition}
For any valuation domain $V$ with maximal ideal $\mathfrak{m}$, we have the following: 
\begin{enumerate}[$(1)$]
\item $\intR(V)$ is faithfully flat over $V$.
\item If $\mathfrak{m}$ is principal with finite residue field, then $\intR(V)$ is not faithfully flat over $V[X]$.
\item If the quotient field of $V$ is algebraically closed, then  $\intR(V)$ is flat over $V[X]$ but not faithfully flat over $V[X]$. 
\end{enumerate}    
\end{proposition}
\begin{proof}
(1) This follows from  Proposition \ref{FFPr}.

(2) Assume that $\mathfrak{m}$ is principal with finite residue field. Then, by \cite[Proposition I.3.16]{C97}, $\mathrm{Int}(V)\neq V[X]$, and hence it follows from the previous corollary that $\intR(V)$ is not faithfully flat over $V[X]$. 

(3) Assume that the quotient field of $V$ is algebraically closed. From \cite[Proposition 2.4]{CL98}, we have $\intR(V)=S^{-1}V[X]$, where $S=\{1+mX;\,m\in\mathfrak{m}\}$. Then, $\intR(V)$ is flat over $V[X]$, and also it is not faithfully flat over $V[X]$ since $S^{-1}V[X]$ is not equal to $V[X]$.
\end{proof}

\begin{corollary}
Let $V$ be a DVR with finite residue field. If the quotient field of $V$ is algebraically closed, then the $V[X]$-module $\intR(V)$ is flat but not faithfully flat. 
\end{corollary}

Now, let us investigate the cyclicity  of the two $D$-modules $\intR(D)$ and  $\mathrm{Int}_B(E,D)$.	 
\begin{proposition}\label{12}
The $D$-module $\intR(D)$ is never cyclic. 
\end{proposition}
\begin{proof} 
By way of contradiction, suppose that $\intR(D)$ is a cyclic $D$-module. Then there exists $\varphi \in \intR(D)$ such that $\intR(D) = \varphi D$. Since $D \subseteq \intR(D)$, there exists $d \in D$ such that $1 = d\varphi$ and then $\varphi$ is an element of $K$. Thus, $\varphi$ is an element of $D$ and it is invertible in $D$ because $\varphi \in \intR(D)$. Therefore, $\intR(D)=D$, which is a contradiction since $\psi(X)=X\in \intR(D)\backslash D$. Consequently, the $D$-module $\intR(D)$ is never cyclic.
\end{proof}
	  
Concerning the $D$-module  $\mathrm{Int}_B(E,D)$, we have:
\begin{proposition}
For any overring $B$ of $D$, the $D$-module $\mathrm{Int}_B(E,D)$ is cyclic if and only if it is equal to $D$. 
\end{proposition}
\begin{proof}
We argue mimicking the  proof of the above proposition to show the direct implication.  The converse is trivial.
\end{proof}
	 
It is clear that for any $d$-ring $D$, both $\mathrm{Int}(D)$ and $\intR(D)$ are simultaneously (locally) free as $D$-modules. Thus, we close our paper with the following open question.
\begin{question} 
Is there any example of a non-$d$-ring $D$ for which the $D$-module $\intR(D)$ is (locally) free or not?
\end{question}

{
\textbf{Acknowledgements.} 
The authors would like to express their sincere thanks to the referee for the rich, fruitful and detailed comments that have greatly helped to improve this paper.}

\end{document}